\documentclass[12pt,oneside]{amsart} 
\usepackage{amssymb,amscd}
\usepackage{euscript}

      \theoremstyle{plain}
      \newtheorem{theorem}{Theorem}[section]
      \newtheorem{lemma}[theorem]{Lemma}
      
      \newtheorem{proposition}[theorem]{Proposition}
      \newtheorem{remark}[theorem]{Remark}
      
      \newtheorem{definition}[theorem]{Definition}      

\numberwithin{equation}{section}

      \makeatletter
      \def\@setcopyright{}
      \def\serieslogo@{}
      \makeatother

\def\M{\EuScript{M}}

\def\R{\mathbb R}

\def\Z{\mathbb Z}

\def\N{\mathbb N}

\def\a{\alpha}

\def\dist{\text{dist}}
\def\diam{\text{diam}}
\def\Id{\text{Id}}
\def\e{\varepsilon}
\def\Ci{C^\infty}

\def\xf{{\lambda}}
\def\xs{{\mu}}

\def\bw{\bar{U}}
\def\f{\bar f}

\def\B{\mathcal B}
\def\G{\mathcal G}

\def\S{\mathcal S}
\def\po{{\mathcal{R}}}
\def\s{{\mathcal{S}}}
\def\c{{\mathcal{C}}}
\def\n{{\mathcal{N}}}
\def\E{{\mathcal{E}}}

\def\V{{\mathcal{V}}}

\def\f{{\mathcal {F}}}
\def\fd{{F}}

\def\p{{\mathcal {P}}}
\def\pd{{P}}

\def\h{{\mathcal {H}}}
\def\hd{H}

\def\QED{\hfill\hfill{\square}}

\begin{document}

\author{Boris Kalinin and Victoria Sadovskaya$^{\ast}$}

\address{Department of Mathematics, The Pennsylvania State University, University Park, PA 16802, USA.}
\email{kalinin@psu.edu, sadovskaya@psu.edu}

\title [Normal forms on contracting foliations]
{Normal forms on contracting foliations: smoothness and homogeneous structure} 

\thanks{$^{\ast}$ Supported in part by NSF grant DMS-1301693}



\begin{abstract}
In this paper we consider a diffeomorphism $f$ of a compact manifold 
$\M$ which contracts  an invariant foliation $W$ with smooth leaves.  
If the differential of $f$ on $TW$ has narrow band spectrum, there exist coordinates $\h _x:W_x\to T_xW$ in which $f|_W$ has polynomial form.
We present a modified approach that allows us to construct maps $\h_x$
that depend smoothly on $x$ along the leaves of $W$. Moreover, we show 
that on each leaf they give a coherent atlas with transition maps in a finite dimensional Lie group. Our results apply, in particular, to $C^1$-small perturbations of algebraic systems.

\end{abstract}

\maketitle


\section{Introduction}

The theory of normal forms of maps plays an important role in dynamics
and goes back to the works of Poincare and Sternberg \cite{St}.
Normal forms at fixed points and invariant manifolds have been extensively 
studied \cite{BK}. More recently, the theory of non-stationary linearizations
and, more generally, normal forms was developed in the context of diffeomorphisms with contracting foliation \cite{KL,GK,G,F2,S}. 
It proved useful in the study of smooth dynamics and rigidity for dynamical 
systems and group actions exhibiting various forms of hyperbolicity,
see for example \cite{KSp97,KS03,KS,Fa,FFH,GKS10,KKtRH10,FKSp10}.  

Let $f$ be a $C^\infty$ diffeomorphism of a compact smooth manifold $\M$, 
and let 
$W$ be a continuous invariant foliation with $C^\infty$ leaves which is contracted 
by $f$, that is $\| Df|_{TW}\|<1$ in some Riemannian metric. The goal of the normal form theory in this setting is
to find a  family of diffeomorphisms $\h_x: W_x \to T_xW$ such that the maps 
\begin{equation} \label{form}
 \tilde f_x =\h_{fx} \circ f \circ \h_x ^{-1}: \;T_x W \to T_{fx}W
\end{equation}
 are as simple  as possible, for example linear maps or polynomial maps in a finite dimensional Lie group. The maps $\h_x$  should depend well on $x$ and ideally form a good atlas on each leaf of $W$.
 For technical reasons,
 it is often easier to operate with marked leaves, which can be identified with the
 tangent spaces $T_xW$, producing an extension of the original system to $TW$.
 The results in \cite{KL,GK,G} are stated in such a setting, but most applications 
come from the foliated systems.

Non-stationary linearization, i.e. existence of $\h_x$ so that $\tilde f_x$ 
in \eqref{form}
are linear,  was first established by Katok and Lewis for one-dimensional $W$  \cite{KL}. For higher-dimensional foliations under a 
uniform 1/2 pinching assumption,  non-stationary linearization 
follows from  results of Guysinsky and Katok
\cite{GK} or  from results of Feres in \cite{F1}, where a differential geometric point of  view was developed. Under a weaker assumption of pointwise 1/2 pinching, it was  obtained in \cite{S}
and additional properties were established in \cite{KS}. These results 
are summarized below.
\vskip.2cm

\noindent {\bf Non-stationary linearization.} 
{\it Suppose that $\| Df|_{TW}\|<1$, and there exist
$C>0$ and $\gamma<1$ such that    for all $x\in \M$ and $n\in\mathbb N$,
 $$ \|\,(Df^n|_{T_xW})^{-1}\,\| \cdot \|\,Df^n|_{T_xW}\,\|^2 
     \leq C\gamma^n.$$
  Then for every $x\in \M$ there exists a $C^\infty$ diffeomorphism 
$\h_x: W_x \to T_xW$ such that
\begin{itemize}
 \item[(i)] $\h_{fx}\circ f \circ \h_x^{-1}=Df|_{T_xW} ,$ 
 \vskip.05cm
 \item[(ii)] $\h_x(x)=0$  and $D_x\h_x$ is the identity map, 
 \vskip.05cm  
 \item[(iii)] $\h_x$ depends continuously on $x\in \M$ in $C^\infty$ topology,
  \vskip.05cm
 \item[(iv)]  Such a family $\h_x$ is unique and 
depends smoothly on $x$ along the leaves of $W$,
 \vskip.05cm
  \item[(v)] The map $\h_y \circ \h_x^{-1}: T_xW\to T_yW$
is affine for any $x\in \M$ and $y\in W_x$. Hence the non-stationary linearization $\h$ defines affine structures on the leaves of $W$.
\end{itemize}
}
\vskip.2cm

For  higher-dimensional $W$ without 1/2 pinching there may be no smooth
non-stationary linearization, and so a polynomial normal form is sought.
Under the {\em narrow band spectrum} assumption, such  forms 
were introduced in \cite{GK,G}. The assumption is satisfied, for example, by perturbations  of algebraic systems. It ensures that the polynomial maps 
involved belong to a finite dimensional Lie group, which is important for applications.
In  \cite{GK,G} Katok and Guysinsky proved existence of  normal form coordinates $\{ \h _x\}$
as well as a centralizer theorem. This was sufficient for some applications, 
such as local rigidity of higher rank actions \cite{KSp97}. However, smooth dependence of $\h _x$ along the leaves of $W$ was obtained under a strong 
extra assumption that 
the splitting of $TW$ into the spectral subspaces is smooth along $W$, 
rather than just continuous. This assumption is not typically satisfied by
perturbations of algebraic systems. There were also no results on the 
coherence of maps $\h _x$ on the leaves of $W$. A geometric point of 
view on normal forms was developed by Feres in \cite{F2}, where he established
existence of an $f$-invariant infinitesimal structure. If this structure, a certain generalized connection, is smooth then it can be integrated to recover normal 
forms. The smoothness, however, again relied on the smoothness of the 
spectral splitting. 

In this paper we overcome these difficulties and obtain a system of normal 
form coordinates $\{ \h _x\}$ which are smooth along the leaves of $W$ and 
on each leaf give a coherent atlas with transition maps in a finite dimensional 
Lie group. This gives homogeneous space structures on the leaves of $W$ 
which are invariant under $f$. Our results hold for any narrow band spectrum 
system without any extra assumptions. In particular, they apply to
 perturbations of algebraic systems. 
We also note that, in contrast to the case of the non-stationary linearization, 
 the system $\{ \h _x\}$ is not unique. In fact, the construction in  \cite{GK,G} produces $\h_x$ which may not be smooth if the spectral splitting in not. 
 We give a modified construction that allows us to obtain the desired properties of $\h_x$. 


\section{Definitions and Results}

Let $\M$ be a smooth compact connected manifold and let $f$ be a $\Ci$
diffeomorphism of $\M$. Let $W$ be an $f$-invariant topological foliation 
of $\M$ with uniformly $\Ci$ leaves. The latter means that all leaves are $\Ci$
submanifolds and that all their derivatives are also continuous transversally to the 
leaves. Slightly more generally, we can consider a  homeomorphism $f$ 
which is uniformly smooth along the leaves of $W$.
We assume that $f$ contracts $W$, i.e. $\| Df |_{TW}\|<1$ in some metric.

Let $\E = TW$ be the tangent bundle of the foliation $W$. 
We denote by $F$ the automorphism of $\E$ given by the
derivative of $f$: 
$$F_x = Df|_{T_xW} : \,\E_x \to \E_{fx}.$$
$F$ induces a bounded linear operator $F^*$ on the space of continuous sections 
of $\E$ by $F^*v(x)=F(v(f^{-1}x))$. The spectrum of complexification of $F^*$
is called {\em Mather spectrum} of $\fd$. Under a mild assumption that 
non-periodic points of $f$ are dense in $\M$, Mather spectrum
consists of finitely many closed annuli centered at $0$, see e.g. \cite{P}.

\begin{definition}  \label{NBdef}
We say that $\fd$ has {\em narrow band spectrum} if 
 its Mather spectrum is contained in a finite union of closed annuli $A_i$, 
 $i=1, \dots, l$, 
bounded by circles of radii $e^{\xf_i}$ and $e^{\xs_i}$, 
where the numbers 
$\xf_1 \le \xs _1 <\dots<\xf_l \le \xs _l<0$ satisfy
\begin{equation} \label{narrow band}
\qquad \xs_i + \xs_l < \xf_i \quad \text{for }\, i=1, \dots, l.
\end{equation}
\end{definition}

This condition  can be written as $\xs_i  - \xf_i   < - \xs_l$ for $i=1, \dots, l,\,$  so it means that the length of each of the intervals 
$[ \xf_i , \xs_i ]$ is smaller than that of $[\xs_l , 0]$. When there is only one spectral interval, \eqref{narrow band} is the uniform 1/2 pinching condition 
which yields non-stationary linearizations and was 
used in \cite{F1} to construct an invariant affine connection.

For the given spectral intervals $\{[\lambda_i, \mu_i]\}$,
 the bundle $\E$ splits into a direct sum 
\begin{equation} \label{splitting}
\E=\E^{1} \oplus \dots \oplus \E^{l}
\end{equation} 
of continuous $F$-invariant sub-bundles so that  Mather spectrum of
$F|_{\E^i}$ is contained in the annulus $A_i$.  
This can be expressed using a convenient metric \cite{GK}:  for each 
$i=1,..., l$ and each $\e >0$ there exists a continuous metric 
$\|.\|_{x,\e}$ on $\E^i$ such that
 \begin{equation} \label{nbrate} 
  e^{\xf _i -\e} \| t \|_{x,\e} \le  \| \fd _x (t)\|_{fx,\e} \le e^{\xs _i +\e} \| t\|_{x,\e}
  \quad\text{for every } t \in \E^i_x.
 \end{equation}

\begin{definition} \label{SRrel}
A  {\em sub-resonance relation} for 
$\,(\xf,\xs)=(\xf_1, \dots, \xf_l,\, \xs_1, \dots, \xs_l )\,$ with $\xf_1 \le \xs _1 <\dots<\xf_l \le \xs _l<0\,$ is a relation of the form 
\begin{equation}\label{sub-resonance}
\xf_i\le\sum s_j \xs_j, \quad\text{where
$s_1,\dots,s_l\,$ are non-negative integers.}
\end{equation}
\end{definition}

Clearly, $s_j=0$ for $j<i$, and $\sum s_j \le \lambda_1/ \mu_l$.
The narrow band condition \eqref{narrow band} implies that if
$s_i\ne 0$, then $s_i=1$ and $s_j=0$ for $j>i$, and hence
all sub-resonance relations are of the form 
$
 \xf_i \le \xs_i \,\text{ or } \,\xf_i\le\sum_{j \ge i+1} s_j \xs_j.
$

For any vector spaces $E$ and $\bar E$ we say that
 a map $P :E \to \bar E$ is polynomial if for some, and hence every, bases of $E$ and $\bar E$  each component of $P$ is a polynomial. A polynomial  map
 $P$ is homogeneous of degree $n$ if $P(a v)=a^n P(v)$ 
 for all $v \in E$
 and $a \in \R$. More generally, for a given splitting $E=E^{1} \oplus \dots \oplus E^{l}$ we say that $P :E \to \bar E$ has homogeneous type $s= (s_1, \dots , s_l)$ if
for any real numbers $a_1, \dots , a_l$ and vectors
$t_j\in E^j$, $j=1,\dots, l,$ we have 
\begin{equation}\label{stype}
P(a_1 t_1+ \dots + a_l t_l)= a_1^{s_1} \cdot ... \cdot a_l^{s_l} \, P( t_1+ \dots + t_l).
\end{equation}
Any polynomial map can be written uniquely as a linear combination of terms
of specific homogeneous types. 

\begin{definition} \label{SRdef}
Suppose $E=E^{1} \oplus \dots \oplus E^{l}$, 
$\bar E=\bar E^{1} \oplus \dots \oplus \bar E^{l}$
and  $P : E \to \bar E$ is  a polynomial map. We split $P$ into
components $P_i : E \to \bar E^i$, $P=(P_1,\dots,P_l)$.
We say that $P$ is of $(\xf,\xs)$
{\em sub-resonance type} if each component $P_i$ has only terms of homogeneous types 
$s= (s_1, \dots , s_l)$  satisfying sub-resonance relations $\xf_i\le\sum s_j \xs_j$.

We denote by $\s^{\xf,\xs}(E, \bar E)$ the space of all polynomials $E\to \bar E$ 
of $(\xf,\xs)$ sub-resonance type.

\end{definition}
It follows from the definition that polynomials in $\s^{\xf,\xs}(E, \bar E)$ have
degree at most 
\begin{equation}\label{degree}
d=d(\xf,\xs)= \lfloor \xf_1/\xs_l \rfloor.
\end{equation}
It was shown in \cite{GK} that in the case of point spectrum, that is  $\xf_i = \xs_i$ for $i=1, \dots ,l$, the elements of $\s^{\xf,\xf}(E,E)$  with $P(0)=0$ and invertible derivative at the origin  form a group $G^{\xf,\xf}(E)$ with respect to composition. More generally, if $(\xf,\xs)$ satisfies the narrow band condition, they generate (under composition) a finite-dimensional Lie group which we denote by $G^{\xf,\xs}(E)$. The maps in $G^{\xf,\xs}$ are called {\em sub-resonance generated} and can be described by adding finitely many 
relations to the set of sub-resonance ones. In fact, $G^{\xf,\xs}(E)$ 
is contained in $G^{\xf',\xf'}(E)$ for a certain $\xf'$, explicitly written in terms 
of $(\xf,\xs)$  \cite{GK}. This larger group may be used in place of $G^{\xf,\xs}(E)$, for simplicity, in all  arguments.

We return to our setting with fixed $F$, $\E$,  and $(\xf,\xs)$. We denote 
by $\s_{x,y}= S^{\xf,\xs}(\E_x,\E_y)$ the space of sub-resonance polynomials
and by $G_x= G^{\xf,\xs}(\E_x)$ the group of sub-resonance generated polynomial maps.  For any $x$ and $y$ in $\M$, any invertible linear map 
$A:\E_y \to \E_x$ that respects the splitting induces an isomorphism 
between the groups $G_x$ and $G_y$. The set $G_{x,y}$ of invertible
sub-resonance generated polynomial maps from $\E_x$ to $\E_y$ is also
naturally defined either by specifying homogeneity types using relations
or simply  by identifying $\E_x$ with $\E_y$: $G_{x,y} =\{ P : \E_x \to \E_y \; |\; A\circ P \in G_x \}$ for any $A$ as above. 
Of course, if the bundle $\E$ and the sub-bundles $\E^i$ are trivial, then 
all $G_x$ and $G_{x,y}$ can be identified with a single group $G^{\xf,\xs}(E)$.

\newpage

\begin{theorem}\label{Main} 
Let $\M$ be a smooth compact connected manifold and let $f$ be a $\Ci$
diffeomorphism of $\,\M$. Let $W$ be an $f$-invariant topological foliation 
of $\M$ with uniformly $\Ci$ leaves. Suppose that $W$ is contracted by $f$, 
i.e. $\| Df |_{TW}\|<1$ for some metric, and that  $Df |_{TW}$ has narrow band spectrum, see Definition \ref{NBdef}.
\vskip.05cm

Then there exist a family $\{ \h_x \} _{x\in \M}$ of $C^\infty$ diffeomorphisms  
 $$\h_x: W_x \to \E_x =T_xW \quad\text{such that }$$

\begin{itemize}
\item[(i)]  $\p_x =\h_{fx} \circ f \circ \h_x ^{-1}:\E_{x} \to \E_{fx}$ is a polynomial map of sub-resonance type for each $x \in \M$,
\vskip.1cm

\item[(ii)]  $\h_x(x)=0$ and $D_x\h_x $ is the identity map  for each $x \in \M$,
\vskip.1cm

\item[(iii)] $\h_x$ depends continuously on $x \in \M$ in $C^\infty$ topology
 and it depends smoothly on $x$ along the leaves of $W$,
 \vskip.1cm 
 
\item[(iv)]  $\h_y \circ \h_x^{-1} : \E_x \to \E_y$ 
is a sub-resonance generated polynomial for each $x \in \M$ and each $y \in W_x$.

 \end{itemize}

\end{theorem}
 
Another way to interpret (iv) is to view $\h_x$ as a coordinate chart on 
$W_x$, identifying it with $\E_x$, see Section \ref{foliations} for more details.
In this coordinate chart, (iv) yields that all transition maps $\h_y \circ \h_x^{-1}$ for $y\in W_x$ are in $\bar G_x$, the 
group generated by $G_x$ and the translations of $\E_x$. Thus $\h_x$ 
gives the leaf a structure of homogeneous space $W_x\sim \bar G_x/G_x$,  
which is consistent with other coordinate charts $\h_y$ for $y\in W_x$ 
and is preserved by the normal form $\p _x$ according to (i).

\begin{remark}
An important and useful feature of normal forms for contractions
was established in \cite{GK}.  Let $g$ be a homeomorphism of $\M$ 
which commutes with $f$,  preserves $W$, and is smooth  along 
the leaves of $W$. Then the maps $\h_x$ bring $g$ to a normal form as well,
i.e. the map $Q_x=\h_{fx} \circ g \circ \h_x ^{-1}$
is a sub-resonance generated polynomial for each $x \in \M$.
\end{remark}




\section{Proof of Theorem \ref{Main}}

\subsection{Proof of (i), (ii), (iii)}

Let us fix a small tubular neighborhood $U$ of the zero section in $TW$
such that for each $x$ in $\M$, we can identify the open set 
$U_x=U \cap T_xW$ in $T_xW$ with a neighborhood of $x$ in the leaf
$W_x$ using the exponential map. The size of this neighborhood can 
be chosen the same for all $x\in \M$.
Thus we can view the diffeomorphism $f$ as a family of maps
$\f _x=f|_{U_x}$ from $U_x \subset T_xW$ to $U_{fx} \subset T_{fx}W$. 
Our first goal is to obtain a formal power series for $\h_x$.

For each map $\f_x : \E_{x} \to \E_{fx}$ we consider its formal power series at $t=0$: 
\begin{equation}\label{f_x}
\f_x(t)= \sum_{n=1}^\infty \fd^{(n)}_x(t).
\end{equation}
As a function of $t$ for a fixed $x$, $\fd^{(n)}_x(t) :\E_{x} \to \E_{fx}$ is a vector valued homogeneous polynomial map of degree $n$ (that is,
in any bases of $\E_x$ and $\E_{fx}$, each coordinate in $\E_{fx}$
is a homogeneous polynomial  of degree $n$ of the coordinates in $\E_x$).

First we construct the formal power series at $t=0$ for the desired coordinate change $\h_x(t)$ as well as the finite power series at $t=0$ 
for the resulting polynomial extension $\p_x(t)$. We use notations for these series similar to \eqref{f_x}:
$$\h_x(t)= \sum_{n=1}^\infty \hd^{(n)}_x(t)\quad \text{and} \quad 
 \p_x(t)= \sum_{n=1}^d \pd^{(n)}_x(t).
 $$
We will use notation  $\fd_x=\fd^{(1)}_x$ for the first derivative of $\f$. 
For $\h$ and $\p$ we choose 
$$\hd^{(1)}_x= \Id : \E_x \to \E_x \quad \text{and} \quad 
\pd^{(1)} _x =\fd _x  \quad \text{for all } x \in \M.
$$

We construct the terms $\hd^{(n)}_x$ inductively to ``eliminate" the part of 
$\fd ^{(n)}_x$ which is not of sub-resonance type. More precisely, we ensure
that the terms $\pd^{(n)}_x$ determined by the conjugacy equation are 
of sub-resonance type. The base of the induction is the linear terms chosen above.
We will obtain $\hd^{(n)}_x$ and $\fd ^{(n)}_x$ which are continuous in 
$x$ on $\M$ and smooth in $x$ along the leaves of $W$. 

Using the notations above we can write the conjugacy equation 
$$ \h_{fx}  \circ \f_{x} =\p_{x} \circ \h_{x} $$ 
as follows:
$$
\left( \Id   +\sum_{k=2}^\infty  H^{(k)} _{fx}  \right) \circ
\left( F_x+\sum_{k=2}^\infty F^{(k)} _{x} \right)  
=\left( F_{x} + \sum_{k=2}^\infty P^{(k)}_{x} \right) \circ
\left( \Id +\sum_{k=2}^\infty H^{(k)}_{x} \right)
$$
for all $x\in \M$. Considering the terms of degree $n$,  we obtain for $n=2$
$$
F^{(2)}_{x}+H^{(2)}_{fx}\circ F_{x} = F_{x} \circ H^{(2)}_{x}+P^{(2)}_{x},
$$
and in general for $n\ge 2$
$$
F^{(n)}_{x}\, + \, H^{(n)}_{fx} \circ F(x)\,+ \,
\sum H^{(k)}_{fx} \circ F^{(j)}_{x}
\,= \,F_{x} \circ H^{(n)}_{x}+ P^{(n)}_{x}+\,
\sum P^{(j)}_{x} \circ H^{(k)}_{x},
$$
where the summations are over all $k$ and $ j$ such that $k j=n$ and $ 1<k,j<n$.
We rewrite the equation as 
\begin{equation}\label{Pn}
F_{x}^{-1} \circ P^{(n)}_{x} =  - H^{(n)}_{x}  + F_{x}^{-1} \circ H^{(n)}_{fx} \circ F_{x} + Q_{x}, \;
\end{equation}
where
$$
Q_{x}= F_{x}^{-1} \left( F^{(n)}_{x} + 
\sum_{kj=n, \;\, 1<k,j<n} H^{(k)}_{fx} \circ F^{(j)}_{x}  
-  P^{(j)}_{x} \circ H^{(k)}_{x} \right).$$
We note that $Q_x$ is composed only of terms $H^{(k)}$ and  $P^{(k)}$ with $1<k<n$, which are already constructed, and terms $F^{(k)}$ with $1<k\le n$,
which are given. Thus by the inductive assumption $Q_x$ is continuous in $x$ on $\M$ and smooth along the leaves of $W$.

We denote by $\po_x^{(n)}$ the vector space of all homogeneous polynomial 
maps of  degree $n$ from $\E_x$ to $\E_x$. Identifying these polynomial maps 
with symmetric $n$-linear maps, one can view this space as $Sym_n(\E_x^*) \otimes \E_x$, where the former is the space symmetric elements in the  $n^{th}$ tensor power of the dual space of $\E_x$,
see \cite{F2}.
Let $\po^{(n)}$ be the vector bundle over $\M$ whose fiber at $x$
 is $\po_x^{(n)}$. We denote by  $\s_x^{(n)}$ and $\n_x^{(n)}$ the 
 subspaces of $\po_x^{(n)}$ consisting of sub-resonance and non sub-resonance  polynomials respectively. These subspaces depend continuously 
 on $x$ and thus  $\po^{(n)}$ splits into the direct sum of the continuous 
 sub-bundles $\s ^{(n)} \oplus \n ^{(n)}$.

Our goal is to find a section $H^{(n)}$ of $\po^{(n)}$ so that the right 
side of \eqref{Pn} is a section of $\s^{(n)}$, and hence so is $P^{(n)}$ when defined by this equation. The sub-bundle $\n ^{(n)}$ in general is only continuous. To construct $H^{(n)}_{x}$ which depends smoothly on $x$ 
along the leaves of $W$ we will work with the factor bundle 
$\po^{(n)} / \s^{(n)}$ rather than with the splitting $\s ^{(n)} \oplus \n ^{(n)}$.
The fiber of $\po^{(n)} / \s^{(n)}$ at $x$ is the vector space 
$\po^{(n)}_x / \s^{(n)}_x$ and, as a continuous vector bundle, 
$\po^{(n)} / \s^{(n)}$ is isomorphic to $\n^{(n)}$ via the natural identification.
A local trivialization of $\po^{(n)} / \s^{(n)}$ can be obtained by fixing locally 
a constant transversal to $\s ^{(n)}$ in any trivialization of $\po^{(n)}$.
We will show in Lemma \ref{smooth s} below that the subspace 
$\s ^{(n)}_x$ depends smoothly on $x$ along the leaves of $W$,
and hence the bundle $\po^{(n)} / \s^{(n)}$ and the projection to it
from $\po^{(n)}$ are smooth along the leaves of $W$.

Projecting  \eqref{Pn} to the factor bundle $\po^{(n)} / \s^{(n)}$, our goal is to solve the  equation
 \begin{equation}\label{barHn}
0 =  - \bar H^{(n)}_{x}  + F_{x}^{-1} \circ \bar H^{(n)}_{fx} \circ F_{x} + \bar Q_{x}, \;
\end{equation}
where $\bar H^{(n)}$ and $ \bar Q$ are the projections of $H^{(n)}$ and $Q$
respectively.

We consider the bundle automorphism $\Phi : \po^{(n)} \to \po^{(n)}$ covering 
$f^{-1}: \M \to \M$ given by
the maps $\Phi_x : \po^{(n)} _{fx} \to \po^{(n)} _{x}$
 \begin{equation}\label{Phi}
\Phi_x (R)= \fd _{x}^{-1} \circ R  \circ \fd _{x}.
 \end{equation}
Since $\fd$ preserves the splitting $\E=\E^1\oplus \dots \oplus \E^l$, it follows from the definition that the sub-bundles $\s^{(n)}$ and $\n^{(n)}$ are $\Phi $-invariant. 
We denote by $\bar \Phi$ the induced automorphism of $\po^{(n)} / \s^{(n)}$
and conclude that \eqref{barHn} is equivalent to 
 \begin{equation}\label{fixedP}
 \bar H^{(n)}_{x} =  \tilde \Phi_x ( \bar H^{(n)}_{fx} ), \quad \text{where }\; \tilde \Phi_x (R)= \bar \Phi_x (R) +  \bar Q_{x}.
\end{equation}
Thus a solution of \eqref{barHn} is a $\tilde \Phi$-invariant section of $\po^{(n)} / \s^{(n)}$. Lemma \ref{contraction} below shows that $\tilde \Phi$ is a contraction and hence has a unique continuous invariant section, which can be explicitly written as
\begin{equation}\label{bar s"}
\bar H _x^{(n)}  = \sum _{k=0}^\infty (F^k_{x})^{-1} \circ \bar Q_{f^k x} \circ  F^k_{x},
\quad \text{where }\,F^k_{x}= F_{f^{k-1}x}\circ \dots \circ F_{fx} \circ F_{x}.
\end{equation}

To show that the maps $\bar H_x^{(n)}$ depend smoothy on $x$ along the 
leaves of $W$ we use the version of the $C^r$ Section Theorem of Hirsch, 
Pugh, and  Shub formulated below. We apply the theorem with 
$\B =\po^{(n)} / \s^{(n)}$ and  $\Psi = \tilde \Phi$. Since $f^{-1}$ expands $W$, 
we have   $\a_x>1$, which  yields that $\bar H _x$ is $C^r$  along the leaves 
of $W$ for any $r$. We conclude that \eqref{bar s"} gives the unique solution 
of \eqref{barHn} which is smooth along the leaves of $W$. Now we take a lift 
$\hd ^{(n)}$ of $\bar \hd ^{(n)}$ to $\po^{(n)}$, which is not unique. We take a 
continuous sub-bundle  $\tilde \n^{(n)}$ which is smooth along the leaves of
 $W$ so that subspace $\tilde \n_x^{(n)}$ is sufficiently close to $\n_x^{(n)}$ 
 and hence is transverse to $\s_x^{(n)}$ for each $x \in \M$.  Then we define 
 $\hd _x^{(n)}$ as the unique element  in $\tilde \n_x^{(n)}$ that projects to 
 $\bar \hd _x^{(n)}$. This lift $\hd ^{(n)}$ is smooth along the leaves of $W$
 and it is the desired section of $\po^{(n)}$ so that the right side of \eqref{Pn} 
 is section of $\s^{(n)}$. Finally, we define $P^{(n)}_{x}$  from \eqref{Pn} and 
 note that it is in $\s_{x,fx}^{(n)}$ and smooth along  the leaves of $W$ since 
 $\hd _x^{(n)}$ is. 

\vskip.2cm

\noindent
{\bf $C^r$ Section Theorem \cite{HPS}.}  {\it  
 Let $f$ be a $C^r$, $r\ge 1$, 
diffeomorphism of a compact smooth manifold $\M$. Let $W$ be an 
$f$-invariant topological foliation with uniformly $C^r$ leaves. 
Let $\B$ be a normed vector bundle over $\M$ and $\Psi :\B \to \B$ be an
extension of $f$ such that both $\B$ and $\Psi $ are uniformly $C^r$ along 
the leaves of $W$.

Suppose that $\Psi $ contracts fibers of $\B$, i.e. for any $x \in \M$ and any $u,w \in \B(x)$ 
\begin{equation} \label{k_x}
   \|\Psi (u)-\Psi (w)\|_{fx} \le k_x \|u-w\|_x \;\text{ with }\, \sup\{\, k_x:\, x \in \M\,\}  <1.
\end{equation}
Then there  exists a unique continuous $\Psi$-invariant  section of $\B$.
Moreover, if 
\begin{equation}
   \sup\{ \,k_x \a_x ^r:\, x \in \M\,\}  <1, \;\text{ where }\,
    \a_x = \|(Df|_{T_xW})^{-1}\|, \label{C^r eq}
\end{equation}
then the unique invariant section is uniformly $C^r$ smooth along the leaves of $W$. 
}
\vskip.2cm

This version of the $C^r$ Section Theorem (see \cite[Theorem 3.7]{KS07}) summarizes for our context Theorems 3.1, 3.2, and  3.5,  and Remarks 1 and 2 after Theorem 3.8 in  \cite{HPS}. In this theorem the smoothness of the invariant section is obtained only along $W$, so the smoothness of  $\B$ and $\Psi $ is required along $W$ only. It also follows from the proof in \cite{HPS} that the contraction in the base 
\eqref{C^r eq} needs to be estimated only along $W$ (formally, this can be
obtained by applying the theorem with the base manifold $\M$ considered 
as the disjoint union of the leaves of $W$). This has been observed in the study of partially hyperbolic systems, see for example Introduction, Theorem 3.1 and remarks
in \cite{PSW}.

\begin{lemma} \label{contraction}
The map $\Phi : \n^{(n)} \to \n^{(n)}$ given by \eqref{Phi} is a contraction
in the sense of \eqref{k_x}, and hence
so is $\,\tilde \Phi : \po^{(n)} / \s^{(n)} \to \po^{(n)} / \s^{(n)}$  given by \eqref{fixedP}.
\end{lemma}

\begin{proof}
We recall that the norm of a homogeneous polynomial map $R: E \to  \bar E$ of 
degree $n$ is defined as $\|R\|=\sup\{\,\|R(v)\|:\; v\in E, \;  \|v\|=1 \,\}$. If $P$ 
is a linear or a homogeneous polynomial map $\tilde E \to E$, then for the 
norm of the composition we have
 \begin{equation}\label{normP}
 \|\, R \circ P \,\|\leq \|R\|\cdot \|P\|^n.
\end{equation}
Suppose that $E=E^1\oplus ... \oplus E^l$ and $\|v\|= \max_i \|v_i\|$ for 
$v=(v_1, ..., v_l)$. If $R$ is of homogeneous type  $s=(s_1,  ..., s_l)$ then 
by \eqref{stype} we get $\|R(v)\| \le \|R\|\, \|v_1\|^{s_1} ... \|v_l\|^{s_l}$.

We will use the norm on $\E$ which is the maximum of the norms on 
$\E^i$ used in \eqref{nbrate}:  
$\|t\|=\|t\|_{x,\e}= \max_i \|t_i\|_{x,\e}$ for any $t=(t_1,..., t_l)\in \E_x$.

Then for a polynomial $R: \E_{f x} \to \E_{f x}^i $ of homogeneous type  
$s=(s_1,  ..., s_l)$ we have 
\begin{equation}\label{normS}
 \|\, R \circ F_x \,\|\leq \|R\| \cdot \| F|_{\E^1_x}\|^{s_1} ... \| F|_{\E^l_x}\|^{s_l}
\end{equation}
 If $R \in \n_{f x}$ then, by definition, we have 
$\xf _i > \sum s_j \xs _j$. We can choose a sufficiently small $\e>0$
so that $  \xf _i  > \sum s_j \xs_j + (n+2)\e$  for all such relations.
It follows from  \eqref{nbrate} that $\| F|_{\E^j_x}\|\le e^{\xs_j +\e}$ 
and $\|F|_{\E^i_x}^{-1}\|\le e^{-\xf_i +\e}$, so we conclude that for each $x \in \M$
$$
\begin{aligned}
\| \Phi_x (R) \| &\le \|F|_{\E^i_x}^{-1}\| \cdot \| R \|  \cdot \prod_j \| F|_{\E^j_x}\|^{s_j} 
\le e^{-\xf_i +\e} \cdot \| R \|  \cdot \prod_j (e^{\xs_j +\e})^{s_j} \le \\
& \le e^{-\xf_i + \sum s_j \xs_j +(n+1)\e} \cdot   \| R \|  
\le   e^{-\e}    \| R \| .
\end{aligned}
$$ 
Thus $\Phi$ is a contraction on $\n^{(n)}$. The second statement follows since the 
linear part $\bar \Phi$ of $\tilde \Phi$ is a contraction. Indeed, $\bar \Phi$
corresponds to $\Phi$ under the identification of $\po^{(n)} / \s^{(n)}$  and 
$\n^{(n)}$ as continuous vector bundles given by the natural linear isomorphisms between $\po^{(n)}_x / \s^{(n)}_x$ and $\n^{(n)}_x$, which depend continuously on the base point. Hence $\|\bar \Phi _x \| \le   e^{-\e}$
with respect to a continuous family of norms on the fibers of 
$\po^{(n)} / \s^{(n)}$.

\end{proof}


\begin{proposition} \label{sr-flag}
The vector space of sub-resonance polynomials $\s(E)=\s^{(\xf,\xs)} (E,E)$ 
for the given $(\xf,\xs)$ and splitting $E= E^1 \oplus \dots \oplus E^l$ 
depends only on the fast flag $V$:
$$
  E^1=V^1 \subset V^2 \subset ... \subset V^l =E, \quad \text{where }\;
  V^i= E^1 \oplus \dots \oplus E^i.
$$
More precisely, $\s(E)=\s(\tilde E)$ if $\tilde E$ is $E$ equipped with
 a splitting that generates the same fast flag: 
 $V^i= \tilde E^1 \oplus \dots \oplus \tilde E^i$, for $i=1,... ,l$.
\end{proposition}
\begin{proof} 
Let $A: E \to \tilde E$ be a linear isomorphism with $A(E^i)= \tilde E^i$, 
for $i=1,... ,l$. Then $R \in \s (E)$ if and only if 
$\tilde R = A \circ R \circ A^{-1} \in \s (\tilde E)$. Also, $A$ and $A^{-1}$ are
{\em block triangular} for the splitting $E=E^1 \oplus \dots \oplus E^l$, that is
their matrices are block triangular in any basis adapted to this splitting or, equivalently, they belong to $\s(E)$. To complete the proof we need to show 
that if $R \in \s (E)$ then so are $A \circ R$ and $R \circ A^{-1}$. 
 
 By splitting $R$ into components and homogeneity types,
it suffices to check this for $R : E \to E^i$ corresponding to a 
sub-resonance relation $\xf _i \le \sum s_j \xs _j$, whose combinatorial  type 
we dente by $(i; s_1,  ..., s_l)$. We call a relation $(m; s_1',  ..., s_l')$
with $|s'|=|s|$ {\em subordinate} to $(i; s_1,  ..., s_l)$ if $m\le i$ and 
 $\sum _{j \le k} s_j' \le \sum _{j \le k} s_j$ for each $k$ 
 (this gives a partial order on relations). It follows that $\xf _m \le \sum s_j' \xs _j$ 
 so that all subordinate relations are also of sub-resonance type. 
 It is easy to see that the homogeneity type of any term in $A \circ R$ or
  $R \circ A^{-1}$ corresponds to a relation subordinate to $(i; s_1,  ..., s_l)$,
 and hence these polynomials are in $\s (E)$.
 
In fact, the subspace spanned by all sub-resonance polynomials corresponding 
to the relations subordinate to $(i; s_1,  ..., s_l)$ can be explicitly described in terms of the flag $V$ only: it consists of all polynomial maps 
$R : \E_x \to V^i_x$ such that for each $k$ the degree of $R$ along the 
subspace $V^k_x$ is at most $\sum _{j \le k} s_j$. This is an alternative way to
see that $\s(E)$ depends only on the fast flag.
\end{proof}

\begin{lemma} \label{smooth s}
The fibers $\s_x ^{(n)}$ of the vector sub-bundle $\s^{(n)} \subset \po^{(n)}$ depend smoothly on $x$ along the leaves of $W$.
\end{lemma}
\begin{proof} We consider the fast flag $\V_x$ in  $\E_x$:
\begin{equation}\label{fastflag}
\E_x^1=\V_x^1 \subset \V_x^2 \subset ... \subset \V_x^l =\E_x, \quad \text{where }\; \V^i_x= \E^1_x \oplus \dots \oplus \E^i_x
\end{equation} 
First we note that the subspaces $\V^i_x$ depend smoothly on $x$ along $W$. This is relatively well-known and is proved by an application of the $C^r$ 
Section Theorem to a suitable graph transform extension, see 
\cite[Proposition 3.9]{KS07} for  an explicit reference. In fact, in our setting,
smoothness of the fast flag along $W$ follows directly from
normal forms since $\h_x$ maps the fast sub-foliations of $W$ to 
linear foliations  of $\E_x$,  see Section \ref{foliations} for details 
(smooth dependence of $\h_x$ on $x$ is not needed for this argument, 
so there is no circular reasoning).

It now follows that there exist a splitting  $ \E_x = \tilde  \E^1_x \oplus \dots \oplus \tilde \E^l_x$ which is smooth in $x$ along $W$ and defines the same flag $\V_x$. 
Then Proposition \ref{sr-flag} implies that $ \s_x ^{(n)}=\tilde \s_x ^{(n)}$, 
the space of all sub-resonance homogeneous polynomials of degree $n$ 
for the same $(\xf, \xs)$ and the new splitting. The latter  clearly depends 
smoothly on $x$ along $W$. 
 \end{proof}

Thus we have constructed the formal series 
$\h_x(t)= \sum_{n=1}^\infty \hd^{(n)}_x(t)$
for the coordinate change and the polynomial map $\p_x(t)= \sum_{n=1}^d \pd^{(n)}_x(t)$, where $d=\lfloor \xf_1/\xs_l \rfloor$. Now we obtain the actual coordinate change function. We fix $N > d$ and conjugate $\f$ by the polynomial coordinate change $\bar \h ^N$
given by
\begin{equation}\label{barH}
 \bar \h^N_x(t)= \sum_{n=1}^N \hd^{(n)}_x(t)\quad\text{and denote}\quad
  \tilde \f_x(t)=\bar \h_{fx}^N \circ \f_{x} \circ (\bar \h_x^N)^{-1}.
\end{equation}
  We note that $\bar \h^N_x(t)$ is a diffeomorphism on a neighborhood 
  $\tilde U_x \subset U_x$ of $0\in\E_x$ since its differential at $0$ 
  is $\Id$, moreover the size of 
  $\tilde U_x$ can be bounded away from $0$ by compactness of $\M$.
  By the construction of $\bar \h^N$, the maps $\tilde \f_x$ and $\p_x $ have the same derivatives at $t=0$ 
  up to order $N$ for each $x \in \M$. 
  We now look for the coordinate change conjugating 
  $\tilde \f_x$ and $\p_x$ in the space 
  $$
  \c_x = \{ R_x \in C^N(\tilde U_x,\E_x) : \;R_x(0)=0, \;D_0 R_x=\Id, \;
  D^{(k)}_0 R_x=0, \; k=2,...,N \}.
  $$
It is a closed affine subspace of $C^N(\tilde U_x,\E_x)$ with a standard norm
  $$
  \|R \|_N=\max \{\,\| D^{(k)}_t R \|: \; t \in U_x, \; 0\le k \le N\}.
  $$
 Each element in $\c_x$ is a diffeomorphism on a neighborhood of $0\in\E_x$.  We will fix $\delta<1$ and then choose $\tilde U_x$ to be convex and small enough so that for any $R_1, R_2$ in $\c_x$ and $R = R_1 -R_2$ we have
\begin{equation}\label{deriv}
 \|D^{(k)} R \|_{0}\le \delta \, \|D^{(N)} R \|_{0} \quad\text{for } 0 \le k<N
 \quad \text{and hence} \quad \|R \|_{N} = \|R^{(N)} \|_{0}.
\end{equation} 
This is possible since $D^{(k)}_0 R=0$  for $0 \le k\le N$ and hence for any $t \in \tilde U_x$ and $0 \le k < N$ we can estimate
$$
\| D^{(k)}_t R \| \le \| t\| \cdot \sup \{ \| D^{(k+1)}_s R \| : \|s\| \le \|t\|\}
\le \diam \, \tilde U_x \cdot \|R^{(k+1)} \|_{0}.
$$
 We consider the space $\S$ of continuous sections of the bundle $\c$
 and equip it with the distance  $\dist(R_1,R_2)=\| R_1 - R_2  \|_N= 
 \sup_x \| (R_1 - R_2)_x \|_N $. We consider the operator
  $$
   T[R]_x= (\p_x)^{-1} \circ R_{fx} \circ \tilde \f_{x}
  $$
on $\S$  so that the fixed point of $T$ is the desired coordinate change.
Note that $T[R]$ is in $\S$ by the definition of $\c_x$ and the coincidence of the 
derivatives of $\p$ and $\tilde \f$ at 0.

  To show that $T$ is a contraction on $\S$ we denote $R = R_1 -R_2$
  for $R_1, R_2$ in $\S$ and estimate  
  $\| T[R] \|_N = \sup_x \| T[R]_x \|_N $, which by \eqref{deriv} equals
 to  $ \sup_x \| T[R]_x^{(N)} \|_0$.
\begin{equation} \label{est}
\begin{aligned}
   D_t^{(N)} T[R]_x & = D_t^{(N)}  \left( (\p_x)^{-1} \circ R_{fx} \circ \tilde \f_{x} \right)= \\
  & = D_{R_{fx}(\tilde \f_{x}(t))}^{(1)} \, (\p_x)^{-1} \circ D_{\tilde \f_{x}(t)}^{(N)} R_{fx} \circ D_t^{(1)}\tilde \f_{x} \,+ \,J,
  \end{aligned}
\end{equation}
where $J$ consists of  a fixed number of terms of the type
$$
   D_{R_{fx}(\tilde \f_{x}(t))}^{(i)} \, (\p_x)^{-1} \circ D_{\tilde \f_{x}(t)}^{(j)} R_{fx} \circ D_t^{(k)}\tilde \f_{x}, \quad ijk=N, \; j<N,
  $$
  whose norm can be estimated using \eqref{normP} as
$$
  \| D_{R_{fx}(\tilde \f_{x}(t))}^{(i)} \, (\p_x)^{-1} \| \cdot \| D_{\tilde \f_{x}(t)}^{(j)} R_{fx} \|^i\cdot \| D_t^{(k)}\tilde \f_{x}\|^{ij}.
  $$  
By \eqref{deriv},  $ \| D_{\tilde \f_{x}(t)}^{(j)} R_{fx}\| \le \delta \, \| D^{(N)} R_{fx} \|_{0}<1$ if $\delta$ is small enough. Therefore, there exists a constant $M=M(\f, \p, N)$ such that
\begin{equation}\label{J}
\| J\| \le M \cdot  \delta \| D^{(N)} R_{fx} \|_{0} 
\le M \delta \cdot \|  R \|_{N}.
\end{equation}

We estimate  the main term in \eqref{est} as follows
\begin{equation} \label{main term}
\begin{aligned}
&\| D_{R_{fx} (\tilde \f_{x}(t))}^{(1)} \, (\p_x)^{-1} \circ D_{\tilde \f_{x}(t)}^{(N)} R_{fx} \circ D_t^{(1)}\tilde \f_{x} \| \,\le \\
&\le \,
  \| D_{R_{fx}(\tilde \f_{x}(t))}^{(1)} (\p_x)^{-1} \| \cdot \| R \|_N \cdot \|  D_t^{(1)}\tilde \f_{x} \|  ^N \,\le \\
 & \le \, e^ {-\lambda_1+\e} \cdot \| R \|_N \cdot e^ {N\mu_l +N\e} \,= \,k'\cdot \| R \|_N,
 \end{aligned}
\end{equation}
where $k'=\exp((N\mu_l - \lambda_1) +(N+1)\e)$.  Since $\lambda_1> N \mu_l$ by the 
choice of $N$ we can take $\e$ small enough so that $k'<1$.
Combining the estimates \eqref{J} and \eqref{main term} we get
$$
\| T[R_1-R_2] \|_N = \sup_x \| T[R_1-R_2]_x^{(N)} \|_0 \le 
(k'+M \delta) \cdot \| R_1-R_2 \|_{N}
$$
so that $T$ is a contraction if $\delta$ is small enough. Thus $T$ has
a unique fixed point, which is a continuous family $\tilde \h^N_x$ of 
coordinate changes conjugating $\p_x$ and $\tilde \f_x$ given by \eqref{barH}. 

We conclude that the maps $\h^N_x=\tilde \h^N_x \circ \bar \h^N_x$ give
a family of $C^N$ diffeomorphisms defined on a neighborhood of the 
zero section of $\E$ which depend continuously on $x\in \M$ in $C^N$ 
topology and satisfy (i) and (ii) of Theorem \ref{Main}. Since $f$ contracts 
$W$,  this family extends uniquely to a family of $C^N$ global 
diffeomorphisms $\h^N_x : W_x \to \E_x$ that satisfy (i). 
Once we fix a formal power series for $\h$, the construction works 
for each $N>d$ and $\h^N$ is unique among $C^N$ diffeomorphisms
whose derivatives up to order $N$ are given by the series.
This means that all $\h^N$ coincide and give a family of $\Ci$ 
diffeomorphisms.

Thus we have proved parts (i),(ii) and (iii) with smoothness along $W$
established so far for the Taylor coefficients of $\h _x$ at $0$. The 
smoothness of $\h _x$ in $x$ along $W$ will follow from this once we 
establish part (iv). Indeed, in the coordinates on $W_x$
given by $\h _x$ for a fixed $x$, for all other $y \in W_x$ the maps $\h_y$ 
are polynomials and thus coincide with their finite Taylor series.


\subsection{Consistency of the fast foliations} \label{foliations}
First we describe some properties of the fast flag \eqref{fastflag} and 
related properties of $\h _x$ and of sub-resonance generated maps.
The leaves of $W$ are subfoliated by unique foliations $U^{k}$ tangent to $V^k_x= \E_x^{1} \oplus \dots \oplus \E_x^{k}.$
We denote by $\bw^{k}$ the corresponding foliations of $\E_x$ obtained by
the identification $\h _x :W_x \to \E_x$. Thus we obtain the foliations of 
$\bw^{k}$ of $\E$ which are invariant under the polynomial normal form
maps $\p_x$. Since the maps $\h _x$ are diffeomorphisms, $\bw^{k}$ are 
also unique fast foliations with the same contraction rates. They are 
characterized by
$$\text{for } \;y, z \in  \E_x  \qquad z \in \bw^{k} (y) \; \Leftrightarrow \; \dist (\p^n _x(y),\p^n_x (z)) \le C e^{n (\xs _k +\e)} \text{ for all } n \in \N$$ 
for any $\e$ sufficiently small so that $\xs _k +\e< \xf_{k+1}$.

It follows from  Definition \ref{SRdef} that sub-resonance polynomials 
$R \in \s_{x,y}$ are {\em block triangular} in the sense that $\E^i$ component 
does not depend on $\E^j$ components for $j<i$ or, equivalently, it maps
the subspaces $\V^i_x$ of fast flag in $\E_x$ to those in $\E_y$. 
By considering compositions, we obtain that any sub-resonance generated polynomial in $G _{x,y}$ is also block triangular.

It is easy to see that all derivatives of a sub-resonance  polynomial are 
sub-resonance polynomials. 
In particular, the derivative 
$D_y \p_x$ at {\em any} point $y \in \E_x$ is sub-resonance and hence is 
block triangular and thus maps subspaces parallel to $V^{k}_x$ to subspaces 
parallel to $V^{k}_{fx}$. Hence the foliation of $\E$ by those parallel to 
$V^{k}_x$ in $\E_x$ is invariant under the maps $\p_x$ and hence coincides 
with $\bw^{k}$ by uniqueness of the fast foliation. 
It follows that for any $x\in \M$ and any $y\in W_x$  the diffeomorphism
  \begin{equation}\label{hdef}
\G_{x,y}:=\h_y \circ \h_x^{-1} : \E_x \to \E_y
\end{equation}
maps the fast flag of linear foliations of $\E_x$ to that of $\E_y$.
In particular, the same holds for its derivative 
$ D_0 \G_{x,y}=D_x H_y : \E_x \to \E_y$ and we conclude that 
$ D_0 \G_{x,y}$ is block triangular and thus is a sub-resonance 
linear map.



\subsection{Proof of (iv): Consistency of normal form coordinates}
In this section we show that the map $\G_{x,y}$ in \eqref{hdef}
is a sub-resonance generated polynomial.  First we note that
$$
\G_{x,y} (0)  =\h_y (x)  =:\bar x \in \E_y \quad\text{and}\quad
 D_0 \,\G_{x,y}=D_x \h_y.
$$ 
Since $ \h_{f^n x}^{-1}  \circ \p ^n_x  \circ \h_{x} = f^n = \h_{f^n y}^{-1}  \circ \p ^n_y  \circ \h_{y}$ we obtain that
$$  \h_{f^n y} \circ \h_{f^n x}^{-1}  \circ \p ^n_x  
= \h_{f^n y} \circ f^n \circ \h_x^{-1} = 
  \p ^n_y  \circ \h_y \circ \h_x^{-1}  \quad \text {and hence}
  $$ 
  \begin{equation}\label{hcom}
\G_{f^n x,f^n y} \circ \p ^n_x =  \p ^n_y  \circ \G_{x,y}.
\end{equation}
Now we consider the formal power series  for $\G_{x,y} : \E_x \to \E_y$ at $t=0 \in \E_x$: 
$$\G_{x,y}(t) \sim G_{x,y}(t)=\bar x + \sum_{m=1}^\infty G^{(m)}_{x,y}(t). 
$$
Our first goal is to show that all its terms are sub-resonance generated.
We proved in Section \ref{foliations} that the first derivative 
$ G^{(1)}_{x,y}=D_x H_y $ is a  sub-resonance linear map. 

Inductively, we assume that $ G^{(m)}_{x,y}$ has only sub-resonance generated
terms for all $x\in \M$, $y \in W_x$, and $m=1,...,k-1$ and show that the same holds for $ G^{(k)}_{x,y}$. We split $G^{(k)}_{x,y}=S_{x,y}+N_{x,y}$ into the 
sub-resonance generated part and the rest. Using invariance under contracting 
maps $\p_x$, it suffices to show that $N_{x,y}= 0$ for all $y \in W_x$ that are sufficiently close to $x$. Assuming the contrary, we fix such $x$ and $y$
with $N_{x,y}\ne 0$. We will write $N_{x}$ for $N_{x,y}$ and 
$N_{f^nx}$ for $N_{f^nx,f^ny}$. Now we consider order $k$ terms in 
the Taylor series at $0\in \E_x$ for \eqref{hcom}. Taylor series  for $\p _x^n$ at 
$0\in \E_x$ coincides with $\p _x^n(t)=\sum_{m=1}^M \pd ^{(m)}_{x}(t)$. 
We also consider the Taylor series  for $\p _y^n$ at $\G_{x,y}(0) = \bar x \in \E_y$ 
$$
\p _y^n(z)= \bar x_n + \sum_{m=1}^M Q^{(m)}_{y}(z-\bar x), \quad 
\text{where } \bar x_n =\p _y^n (\bar x) 
$$
All terms $Q^{(m)}$ are sub-resonance as the derivatives of a sub-resonance  polynomial.
Consider the formal power series  for 
$$\G_{f^n x,f^n y}(t) \sim G_{f^n x,f^n y}(t) = \bar x_n  + \sum_{m=1}^\infty G^{(m)}_{f^n x,f^n y}(t).
$$
Now we obtain from \eqref{hcom}
$$\bar x_n  + \sum_{j=1}^\infty G^{(j)}_{f^n x,f^n y}\left( \sum_{m=1}^M \pd ^{(m)}_{x}(t) \right) = \,\bar x_n  + 
\sum_{m=1}^M Q^{(m)}_{y}\left( \sum_{j=1}^\infty G^{(j)}_{x,y}(t)  \right).
$$
Since any composition of sub-resonance generated terms is sub-resonance generated, taking non sub-resonance generated terms of order $k$ in the 
above equation yields
 \begin{equation}\label{NRcom}
N_{f^n x}\left(  \pd ^{(1)}_{x}(t) \right) =  Q^{(1)}_{y}\left( N_{x}(t)  \right).
\end{equation}
We decompose $N$ into components $N=(N^1,... , N^l)$ and let $i$ be the largest index so that  $N^i _x \ne 0$, i.e. there exists $t' \in \E_x$ so that $z'= N^i (t')$
 has non-zero component in $\E^i_y$. We denote
  \begin{equation}
 w= Q^{(1)}_{y}(z') =  D_{\bar x} \p _y^n (z') \in \E_{f^n y} \quad  \text{ and let }\; 
 w_i =Pr_i (w) \in \E^i_{f^n y} 
 \end{equation}
 be  its $i$ component. We claim that 
 \begin{equation}\label{right}
 \| w_i \|  \ge C e^{n (\xf _i -\e)}
 \end{equation}
where the constant $C$ does not depend on $n$. This follows from \eqref{nbrate} and the fact that 
$$D_{\bar x} \p _y^n = D_{f^n x}H_{f^n y} \circ \fd^n_x \circ (D_x H_y) ^{-1} . 
$$ 
Indeed, the differentials $D_{f^n x}H_{f^n y}$ and $D_{x}H_y ^{-1} $ preserve the fast flag and are close to identity since $x$ is close to $y$ by our assumption
and hence $f^n x$ is close to $f^n y$. Then $z=(D_x H_y) ^{-1} (z')$ has non-zero component $z_i$ in $\E^i_x$, which is transformed by $ \fd^n_x$ according to 
\eqref{nbrate}: $\|  \fd^n_x (z_i) \| \ge e^{n (\xf _i -\e)} \| z_i \|$. Then finally
 $\| w_i \|  \ge C' \|  \fd^n_x (z_i) \|$ as $D_{f^n x}H_{f^n y}$ is close to identity
 and $\fd^n_x (z)$ has no $j$ components for $j>i$.
 \vskip.1cm
 
 Now we estimate from above the $i$ component of the left side of \eqref{NRcom} at $t'$. 
 First, 
 $$\| \pd ^{(1)}_{x}(t'_j) \| = \| \fd^n_x (t'_j) \|\le  \| (t') \| e^{n  (\xs _j +\e)}
 \quad\text{for any }j .$$
Let $N^s_{f^n x}$ be a  term of homogeneity type $s=(s_1,  ..., s_l)$ in the component $N_{f^n x}^i$.  We estimate as in \eqref{normS}
 $$
 \| N^s_{f^n x} \left(  \pd ^{(1)}_{x}(t')\right) \| \le 
\| N_{f^n x}\| \cdot  \| (t') \|^k \cdot  e^{n \sum s_j (\xs _j +\e)} 
$$
Since no term in $N_{f^n x}^i$ is a sub-resonance one, we have $\xf _i > \sum s_j \xs _j$.  We can choose a sufficiently small $\e>0$
so that $  \xf _i  > \sum s_j \xs_j + (n+2)\e$  for all such relations.
Hence the left side of \eqref{NRcom}  at $t'$ can be estimated as
$$
\| N^s_{f^n x} \left(  \pd ^{(1)}_{x}(t')\right) \| \le C'  e^{n (\xf _i -2\e)}, 
$$
where the constant $C'$ does not depend on $n$ since the norms of 
$G^{(k)}_{f^n x,f^n y}$, and hence those of $N_{f^n x}$, are uniformly bounded.
This contradicts \eqref{NRcom} and \eqref{right} for large $n$.
\vskip.1cm

Thus we shown that the Taylor series  $G_{x,y}$  of $\,\G_{x,y}$  at $0$ contains
only sub-resonance generated terms, and in particular it is a finite polynomial. 
It remains to show that $\G_{x,y}$ coincides with its Taylor series.

In addition to \eqref{hcom} we have the same relation for their finite Taylor
series
 \begin{equation}\label{Hcom}
G_{f^n x,f^n y} \circ \p ^n_x =  \p ^n_y  \circ G_{x,y}
\end{equation}
and hence denoting $\Delta_n= \G_{f^n x,f^n y}-G_{f^n x,f^n y}$
we obtain
 \begin{equation}\label{Hcom2}
 \Delta_0 = (\p ^n_y ) ^{-1} \circ \Delta_n \circ \p ^n_x   
\end{equation}
 Below we show that the right hand side of \eqref{Hcom2}  tends to 0 as $n \to \infty$ and thus $\G_{x,y}=G_{x,y}$, completing the proof.
\vskip.05cm 
Since $\Delta_n$ is infinitely flat at $0\in \E_{f^n x}$, for each $k$ there exists
$C_{k,\delta}$ such that 
\begin{equation}\label{Ck,delta}
\| \Delta_n (z)\| \le C_{k,\delta} \| z\| ^k \quad\text{for all }
z \in  \E_{f^n x}  \text{ with } \| z\| \le \delta.
\end{equation}
We can choose the constant $C_{k,\delta}$ independent of $n$ since $\G_{x,y}$ depends continuously on $(x,y)$ in $C^\infty$ topology and $f^n x$ remains close to $f^n y$. We fix $k$ and $\e$ such that 
\begin{equation} \label{k,e}
  k \xs _l   +2(k+1)\e < \xf _1.
  \end{equation}
We  estimate  $ (\p ^n_y ) ^{-1}$ and  $ \p ^n_x $ as follows:
\begin{equation}\label{P est}
\| (\p ^n_y ) ^{-1} (t) \| \le e^{-n (\xf _1 -2\e)} \| t\|  \quad \text{and} \quad
\| (\p ^n_x ) (t) \| \le e^{n  (\xs _l +2\e)} \| t\| 
\end{equation}
for any $n \in \N$ provided that sizes of both the input $\| t\|$ and the outputs
$\|(\p ^n_y ) (t) \|$ and $\|(\p ^n_y ) ^{-1} (t) \|$ are at most $\delta$. 
Both estimates follow  from the fact that for any $x \in \M$ 
$$ \| D_0 \p _x  \| = \| \fd _x  \| \le e^{\xs _l +\e} \quad \text{and} \quad
\| D_0 (\p _x  ) ^{-1}   \| = \| (\fd _x)  ^{-1}   \| \le e^{-\xf _1 +\e} 
$$
by \eqref{nbrate} and hence we have similar estimates for derivatives 
at all points $t \in \E_x$ with $\| t\| \le \delta$ for sufficiently small $\delta >0$:
$$ \| D_t \p _x  \| \le e^{\xs _l +2\e} \quad \text{and} \quad 
\| D_t (\p _x  ) ^{-1}   \|  \le e^{-\xf _1 +2\e}. 
$$
Then the bound for $\| (\p ^n_x ) (t) \|$ follows by simply estimating 
its derivative as the trajectory product, as long as all points involved
are $\delta$ close to $0$.

Finally using \eqref{Hcom2}, \eqref{Ck,delta}, and \eqref{P est}  we estimate
$$
\begin{aligned}
\| \Delta_0 (t) \| &\,\le \,
\| (\p ^n_y ) ^{-1} \|\cdot \| \Delta_n \circ \p ^n_x (t)\| \,\le \,
e^{-n (\xf _1 -2\e)} C_{k,\delta} \left( e^{n  (\xs _l +2\e)} \| t\| \right)^k \,=\\
&\,= \,   e^{n \left (-\xf _1 + k \xs _l   +2(k+1)\e \right)} \,C_{k,\delta}\, \| t\| ^k
\, \to 0 \quad\text{as } n\to\infty
\end{aligned}
$$
by \eqref{k,e}. Thus, $\Delta_0=0$, i.e. 
the map $\G_{x,y}$ coincides with its Taylor series.
\vskip.05cm

This completes the proof of Theorem \ref{Main}
$\QED$




\end{document}